\documentclass[a4paper,11pt]{amsart}
\pdfoutput=1

\usepackage{mathtools}
\mathtoolsset{showonlyrefs=true}

\usepackage{amsmath,amssymb,amsthm,amsfonts, mathrsfs, fancyhdr}
\usepackage[usenames,dvipsnames,usenames,dvipsnames,svgnames,table]{xcolor}
\usepackage[expansion=false]{microtype}
\usepackage[all]{xy}

\usepackage{upgreek}

\usepackage{enumitem}
\makeatletter
\newcommand{\mylabel}[2]{#2\def\@currentlabel{#2}\label{#1}}
\makeatother

\setlist[description]{leftmargin=*}
\usepackage{amscd}
\usepackage[active]{srcltx}
\usepackage{verbatim}
\usepackage{epsfig,graphicx,color}
\usepackage{graphicx}
\usepackage{mathtools,xparse}
\usepackage[english]{babel}

\usepackage{pgf, tikz, tikz-cd}
\usepackage{dsfont}
\usetikzlibrary{matrix}

\usepackage{braids}

\usepackage{capt-of}
\usepackage{float}

\usepackage{verbatim}
\usepackage{cite}
\usepackage{manyfoot}

\usepackage{palatino,hhline, bbm}

\usepackage[left=2.7cm,right=2.7cm,top=3cm,bottom=3cm]{geometry}

\usepackage[unicode,bookmarks, pdftex, backref = page]{hyperref}
\definecolor{newblue}{RGB}{0,102,204}
\hypersetup{colorlinks=true,citecolor=newblue,linkcolor=newblue, urlcolor=teal,
pdfpagemode=UseNone, breaklinks=true}

\newtheorem{theorem}{Theorem}[section]
\newtheorem{proposition}[theorem]{Proposition}
\newtheorem{lemma}[theorem]{Lemma}

\newtheorem{question}[theorem]{Question}

\theoremstyle{definition}
\newtheorem{remark}[theorem]{Remark}
\newtheorem{example}[theorem]{Example}
\newtheorem{definition}[theorem]{Definition}

\renewcommand{\to}{\longrightarrow}
\renewcommand{\t}{\mathsf{t}}
\renewcommand{\r}{\mathsf{r}}
\newcommand{\z}{\mathsf{z}}
\newcommand{\x}{\mathsf{x}}
\renewcommand{\S}{\mathfrak{S}}

\usepackage{epigraph}
\title[Double Kodaira structures]{Diagonal double Kodaira structures on finite groups}

\date{}

\author[Francesco Polizzi]{Francesco Polizzi}
\address{Dipartimento di Matematica e Informatica
\newline\indent
Universit\`a della Calabria
\newline\indent
Ponte Pietro Bucci 30B, I-87036 Arcavacata di Rende, Cosenza, Italy}
\email{francesco.polizzi@unical.it}



\begin{document}


\begin{abstract}
We introduce some special presentations on finite groups, that we call \emph{diagonal double Kodaira structures} and whose existence is equivalent to the existence of some special Kodaira fibred surfaces, that we call \emph{diagonal double Kodaira fibrations}.  This allows us to rephrase in  purely algebraic terms some results about finite Heisenberg groups, previously obtained in \cite{CaPol19}, and makes possible to extend them to the case of arbitrary extra-special $p$-groups. 
\end{abstract}

\maketitle





\section{Introduction} \label{sec:intro}

A \emph{Kodaira fibration} is a smooth, connected holomorphic fibration $f_1 \colon S \to B_1$, where $S$ is a compact complex surface and $B_1$ is a compact complex curve, which is not isotrivial (this means that not all its fibres are biholomorphic to each others). The genus $b_1:=g(B_1)$ is called the \emph{base genus} of the fibration, whereas the genus $g:=g(F)$, where $F$ is any fibre, is called the \emph{fibre genus}. If a surface $S$ is the total space of a Kodaira fibration, we will call it a \emph{Kodaira fibred surface}; it is possible to prove that every such a surface is minimal and of general type.

Examples of Kodaira fibrations were originally constructed in \cite{Kod67, At69} in order to show that, unlike the topological Euler characteristic, the signature $\sigma$ of a real manifold is not multiplicative for fibre bundles. In fact, every Kodaira fibred surface $S$ satisfies $\sigma(S) >0$, see for example the introduction of \cite{LLR17}, whereas $\sigma(B_1)=\sigma(F)=0$, and so $\sigma(S) \neq \sigma(B_1)\sigma(F)$.

A \emph{double Kodaira surface} is a compact complex surface $S$, endowed with a \emph{double Kodaira fibration}, namely a surjective, holomorphic map $f \colon S \to B_1 \times B_2$ yielding, by composition with the natural projections, two Kodaira fibrations $f_i \colon S \to B_i$, $i=1, \,2$.

In \cite{CaPol19} the author (in collaboration with A. Causin) introduced a new topological method to construct double Kodaira fibrations, based on the so-called \emph{Heisenberg covers} of $\Sigma_b \times \Sigma_b$, where $\Sigma_b$ denotes a real, closed, connected, orientable surface of genus $b$ (from now on, we will simply write ``a real surface of genus $b$"). These are finite Galois covers $S \to \Sigma_b \times \Sigma_b$, whose branch locus is the diagonal $\Delta \subset \Sigma_b \times \Sigma_b$ and whose Galois group is isomorphic to a finite Heisenberg group. In this note we rephrase the group-cohomological methods of \cite{CaPol19} in a purely algebraic way, by introducing the so-called \emph{diagonal double Kodaira structures} on a finite group $G$, see Definition \ref{def:dks}. These are special presentations of $G$, whose existence is equivalent to the fact that $G$ is a good quotient of some higher genus pure braid group on two strands, where ``good" means that the natural braid called $A_{12}$ in \cite{GG04} has non-trivial image under the quotient map. The existence of diagonal double Kodaira structures yields in turn the existence of some special double Kodaira fibrations, that we call \emph{of diagonal type}, see Definition \ref{def:diagonal-double-kodaira}.  

With this new and compact terminology, we give a short account of some of the main results contained in \cite{CaPol19}, namely
\begin{itemize}
\item the existence of (double) Kodaira fibrations over every curve of genus $b$ (and not only over special curves with extra automorphisms) and the proof that the number of such fibrations over a fixed base can be arbitrarily large, see Theorem \ref{thm:lim-sup};
\item the first ``double solution" to a problem, posed by Geoff Mess, from Kirby's problem list in low-dimensional topology, see Theorem \ref{thm:Kirby};
\item the existence of an infinite family of (double) Kodaira fibrations with slope strictly higher than $3+1/3$, see Theorem \ref{thm:slope} and Remark \ref{rmk:record-slope}.
\end{itemize} 
This paper also contains some new  results, namely 
\begin{itemize}

\item the construction of  diagonal double Kodaira structures on extra-special $p$-groups of any exponent, see Theorems  \ref{thm:extra-special-sdks} and \ref{thm:non-strong}. This extends the equivalent statements for extra-special $p$-groups of exponent $p$ proved in \cite{CaPol19};
\item an explicit upper bound for the slope of a diagonal double Kodaira fibration, see Proposition \ref{prop:slope-diagonal} and Remark \ref{rmk:6-4sqrt2}. 
 \end{itemize} 

An intriguing problem is the existence of diagonal double Kodaira structures on finite groups that are not extra-special or, more generally, on finite groups whose nilpotence class is at least $3$, cf. Remark \ref{rmk:commutators-in-the-center}. However, we will not develop this point here, hoping to come back on it in a sequel to this paper.

\medskip 
\indent \textbf{Acknowledgments.} The author was partially supported by GNSAGA-INdAM. He thanks Andrea Causin for useful comments and remarks. He is also  grateful to the organizers of the 12\emph{th} \emph{ISAAC Congress}, Universidade de Aveiro (2019), and in particular to Alexander Schmitt, for  the invitation and the hospitality. Finally, he is indebted with Geoff Robinson and Derek Holt for their generous help with the proof of Lemma \ref{lem:identify-heisenberg}, see the MathOverflow thread 

\href{https://mathoverflow.net/questions/351163/direct-proof-or-reference-that-a-given-p-group-is-extra-special}{https://mathoverflow.net/questions/351163}, 

\noindent and to 
the anonymous referee for constructive criticism and precise suggestions that considerably improved the presentation of these results.

\medskip
\indent \textbf{Notation and conventions.} 
The order of a finite group $G$ is denoted by $|G|$. If $x \in G$, the order of $x$ is denoted by $o(x)$. The subgroup generated by $x_1, \ldots, x_n \in G$ is denoted by $\langle x_1, \ldots, x_n \rangle$.
The center of $G$ is denoted by $Z(G)$. If $x,y \in G$,
their commutator is defined as $[x,y]=xyx^{-1}y^{-1}$.
We denote both the cyclic group of order $p$ and the field with $p$ elements  by $\mathbb{Z}_p$.


\section{Diagonal double Kodaira structures} \label{sec:dks}

Let $G$ be a finite group and let $b, \, n \geq 2$ be two positive integers.

\begin{definition} \label{def:dks}
A \emph{diagonal double Kodaira structure of type} $(b, \, n)$ on $G$ is a set of  $4b+1$ generators
\begin{equation} \label{eq:dks}
\S = \{\r_{11}, \, \t_{11}, \ldots, \r_{1b}, \, \t_{1b}, \; \r_{21}, \, \t_{21}, \ldots, \r_{2b}, \, \t_{2b}, \; \z \},
\end{equation}
with $o(\z)=n$, such that the following relations are satisfied. We systematically use the commutator notation in order to indicate the conjugacy action, writing for instance $[x, \, y]=zy^{-1}$ instead of $xyx^{-1}=z$.
\begin{itemize}
\item {Surface relations}
\begin{align} \label{eq:presentation-0}
 & [\r_{1b}^{-1}, \, \t_{1b}^{-1}] \, \t_{1b}^{-1} \, [\r_{1 \,b-1}^{-1}, \, \t_{1 \,b-1}^{-1}] \, \t_{1\,b-1}^{-1} \cdots [\r_{11}^{-1}, \, \t_{11}^{-1}] \, \t_{11}^{-1} \, (\t_{11} \, \t_{12} \cdots \t_{1b})=\z \\
& [\r_{21}^{-1}, \, \t_{21}] \, \t_{21} \, [\r_{22}^{-1}, \, \t_{22}] \, \t_{22}\cdots   [\r_{2b}^{-1}, \, \t_{2b}] \, \t_{2b} \, (\t_{2b}^{-1} \, \t_{2 \, b-1}^{-1} \cdots \t_{21}^{-1}) =\z^{-1}
\end{align}
\item {Conjugacy action of} $\r_{1j}$
\begin{align} \label{eq:presentation-1}
[\r_{1j}, \, \r_{2k}]& =1  &  \mathrm{if} \; \; j < k \\
[\r_{1j}, \, \r_{2j}]& = 1 & \\
[\r_{1j}, \, \r_{2k}]& =\z^{-1} \, \r_{2k}\, \r_{2j}^{-1} \, \z \, \r_{2j}\, \r_{2k}^{-1} \; \;&  \mathrm{if} \;  \; j > k \\
& \\
[\r_{1j}, \, \t_{2k}]& =1 & \mathrm{if}\;  \; j < k \\
[\r_{1j}, \, \t_{2j}]& = \z^{-1} & \\
[\r_{1j}, \, \t_{2k}]& =[\z^{-1}, \, \t_{2k}]  & \mathrm{if}\;  \; j > k \\
& \\
[\r_{1j}, \,\z]& =[\r_{2j}^{-1}, \,\z] &
\end{align}
\item {Conjugacy action of} $\t_{1j}$
\begin{align} \label{eq:presentation-3}
[\t_{1j}, \, \r_{2k}]& =1 & \mathrm{if}\;  \; j < k \\
[\t_{1j}, \, \r_{2j}]& = \t_{2j}^{-1}\, \z\, \t_{2j} & \\
[\t_{1j}, \, \r_{2k}]& =[\t_{2j}^{-1},\, \z] \; \; & \mathrm{if}  \;\; j > k \\
& \\
[\t_{1j}, \, \t_{2k}]& =1 & \mathrm{if}\; \; j < k \\
[\t_{1j}, \, \t_{2j}]& = [\t_{2j}^{-1}, \, \z] & \\
[\t_{1j}, \, \t_{2k}]& =\t_{2j} ^{-1}\,\z\, \t_{2j}\, \z^{-1}\, \t_{2k}\,\z\, \t_{2j} ^{-1}\,\z^{-1}\, \t_{2j}\,\t_{2k}^{-1} \; \;  & \mathrm{if}\;  \; j > k \\
&  \\
[\t_{1j}, \,\z]& =[\t_{2j}^{-1}, \,\z] &
\end{align}
\end{itemize}

\begin{remark} \label{inverse-actions}
From \eqref{eq:presentation-1} and \eqref{eq:presentation-3} we can deduce the corresponding conjugacy actions of $\r_{1j}^{-1}$ and $\t_{1j}^{-1}$. We leave the cumbersome but standard computations to the reader.
\end{remark}

\begin{remark} \label{rmk:no-abelian-dks}
Abelian groups admit no diagonal double Kodaira structures. Indeed, the relation $[\r_{1j}, \, \t_{2j}]=\z^{-1}$  in \eqref{eq:presentation-1} provides a non-trivial commutator in $G$, because $o(\z)=n$.
\end{remark}
\end{definition}

\begin{remark} \label{rmk:commutators-in-the-center}
Assume that the commutator subgroup $[G, \ G]$ is contained in the center $Z(G)$, i.e., that $G/Z(G)$ is abelian (being $G$ non-abelian, this is equivalent to the fact that $G$ has nilpotence class $2$, see \cite[p. 22]{Is08}). Then the relations defining a diagonal double Kodaira structure on $G$ assume the following simplified form. 
\begin{itemize}
\item Relations expressing the centrality of $\z$ 
\begin{equation} \label{eq:presentation-01-simple}
[\r_{1j}, \z]=[\t_{1j}, \z]=[\r_{2j}, \z]=[\t_{2j}, \z]=1
\end{equation}
\item {Surface relations}
\begin{align} \label{eq:presentation-0-simple}
 & [\r_{1b}^{-1}, \, \t_{1b}^{-1}] \, [\r_{1 \,b-1}^{-1}, \, \t_{1 \,b-1}^{-1}] \, \cdots [\r_{11}^{-1}, \, \t_{11}^{-1}] \,=\z \\
& [\r_{21}^{-1}, \, \t_{21}] \,  [\r_{22}^{-1}, \, \t_{22}] \cdots   [\r_{2b}^{-1}, \, \t_{2b}] =\z^{-1}
\end{align}
\item {Conjugacy action of} $\r_{1j}$
\begin{align} \label{eq:presentation-1-simple}
[\r_{1j}, \, \r_{2k}]& =1  &  \mathrm{for \; all} \; \; j, \, k \\
[\r_{1j}, \, \t_{2k}]& = \z^{-\delta_{jk}} & \\
\end{align}
\item {Conjugacy action of} $\t_{1j}$
\begin{align} \label{eq:presentation-3-simple}
[\t_{1j}, \, \r_{2k}]& = \z^{\delta_{jk}} & \\
[\t_{1j}, \, \t_{2k}]& =1 & \mathrm{for \; all} \; \; j, \, k \\
\end{align}
\end{itemize}
where $\delta_{jk}$ stands for the Kronecker symbol. 
\end{remark}
\bigskip
If $\S$ is a diagonal double Kodaira structure of type $(b, \, n)$ on $G$, then the subgroup 
\begin{equation} \label{eq:K1}
K_2:=\langle  \r_{21}, \, \t_{21}, \ldots, \r_{2b}, \, \t_{2b}, \; \z  \rangle
\end{equation}
is normal in $G$ and so there is a short exact sequence 
\begin{equation}
1 \to K_2 \to G \to Q_1 \to 1,
\end{equation}
where the elements $\r_{11}, \, \t_{11}, \ldots, \r_{1b}, \, \t_{1b}$ yield a complete set of representatives for $Q_1$. On the other hand, the set of relations defining $\S$ is invariant under the involution
\begin{equation} \label{eq:substitutions-S}
\z \longleftrightarrow \z^{-1}, \quad \t_{1j} \longleftrightarrow \t_{2\; \; b+1-j}^{-1}, \quad  \r_{1j} \longleftrightarrow \r_{2\;\; b+1-j},
\end{equation}
hence we can also see $G$ as the middle term of a short exact sequence
\begin{equation}
1 \to K_1 \to G \to Q_2 \to 1,
\end{equation}
where
\begin{equation*}
K_1:=\langle  \r_{11}, \, \t_{11}, \ldots, \r_{1b}, \, \t_{1b}, \; \z  \rangle
\end{equation*}
and $\r_{21}, \, \t_{21}, \ldots, \r_{2b}, \, \t_{2b}$ yield a complete set of representatives for $Q_2$.

\begin{definition} \label{def:sdks}
A diagonal double Kodaira structure $\S$ as above will be called \emph{of strong type} $(b, \, n)$ if $K_1=K_2=G$. Otherwise, it will be called \emph{of non-strong type} $(b, \, n)$.
\end{definition}
Sometimes we will not specify the pair $(b, \, n)$, and we will simply say that $\S$ is ``of strong type" or ``of non-strong type", respectively. 


\section{The case of extra-special $p$-groups} \label{sec:extra-special}

The following classical definition can be found, for instance, in \cite[p. 183]{Gor07}
and \cite[p. 123]{Is08}.
\begin{definition} \label{def:extra-special}
Let $p$ be a prime number. A finite $p$-group $G$ is called \emph{extra-special} if its center $Z(G)$ is cyclic of order $p$ and the quotient $V=G/Z(G)$ is a non-trivial, elementary abelian $p$-group.
\end{definition}

An elementary abelian $p$-group is a finite-dimensional vector space over the field $\mathbb{Z}_p$, hence it is of the form $V=(\mathbb{Z}_p)^{\dim V}$ and $G$ fits into a short exact sequence
\begin{equation} \label{eq:extension-extra}
1 \to \mathbb{Z}_p \to G \to V \to 1.
\end{equation}
Note that, $V$ being abelian, we must have $[G, \, G]=\mathbb{Z}_p$, namely the commutator subgroup of $G$ coincides with its center. Furthermore, since the extension \eqref{eq:extension-extra} is central, it cannot be split, otherwise $G$ would be isomorphic to the direct product of the two abelian groups $\mathbb{Z}_p$ and $V$, which is impossible because $G$ is non-abelian.  It can be also proved that, if $G$ is extra-special, then  $\dim V$ is even, so $|G|=p^{\dim V +1}$ is an odd power of $p$.  


For every prime number $p$, there are precisely two isomorphism classes $M(p)$, $N(p)$ of non-abelian groups of order $p^3$, namely
\begin{equation*}
\begin{split}
M(p)& = \langle \r, \, \t, \, \z \; | \; \r^p=\t^p=1, \, \z^p=1, [\r, \, \z]=[\t,\, \z]=1, \, [\r, \, \t]=\z^{-1} \rangle \\
N(p)& = \langle \r, \, \t, \, \z \; | \; \r^p=\t^p=\z, \, \z^p=1, [\r, \, \z]=[\t,\, \z]=1, \, [\r, \, \t]=\z^{-1} \rangle \\
\end{split}
\end{equation*}
and both of them are in fact extra-special, see \cite[Theorem 5.1 of Chapter 5]{Gor07}.

If $p$ is odd, then the groups $M(p)$ and $N(p)$ are distinguished by their exponent, which equals $p$ and $p^2$, respectively. If $p=2$, the group $M(p)$ is isomorphic to the dihedral group $D_8$, whereas $N(p)$ is isomorphic to the quaternion group $Q_8$. 

The classification of extra-special $p$-groups is provided by the result below, see \cite[Section 5 of Chapter 5]{Gor07}.
\begin{proposition} \label{prop:extra-special-groups}
If $b \geq 2$ is a positive integer and $p$ is a prime number, there are exactly two isomorphism classes of extra-special $p$-groups of order $p^{2b+1}$, that can be described as follows.  
\begin{itemize}
\item The central product $\mathsf{H}_{2b+1}(\mathbb{Z}_p)$ of $b$ copies of $M(p)$, having presentation
\begin{equation}
\begin{split}
\mathsf{H}_{2b+1}(\mathbb{Z}_p) = \langle \, & \mathsf{r}_1, \, \mathsf{t}_1, \ldots, \mathsf{r}_b,\, \mathsf{t}_b, \, \z \; |  \; \mathsf{r}_{j}^p = \mathsf{t}_{j}^p=\mathsf{z}^p=1, \\
& [\mathsf{r}_{j}, \, \mathsf{z}]  = [\mathsf{t}_{j}, \, \mathsf{z}]= 1, \\
& [\mathsf{r}_j, \mathsf{r}_k]= [\mathsf{t}_j, \mathsf{t}_k] = 1, \\
& [\mathsf{r}_{j}, \,\mathsf{t}_{k}] =\mathsf{z}^{- \delta_{jk}} \, \rangle.
\end{split}
\end{equation}
If $p$ is odd, this group has exponent $p$. 
\item The central product $\mathsf{G}_{2b+1}(\mathbb{Z}_p)$ of $b-1$ copies of $M(p)$ and one copy of $N(p)$,  having presentation
\begin{equation}
\begin{split}
\mathsf{G}_{2b+1}(\mathbb{Z}_p) = \langle \, & \mathsf{r}_1, \, \mathsf{t}_1, \ldots, \mathsf{r}_b,\, \mathsf{t}_b, \, \z \; | \;  \mathsf{r}_{b}^p = \mathsf{t}_{b}^p=\mathsf{z}, \\ 
&\mathsf{r}_{1}^p = \mathsf{t}_{1}^p= \ldots = \mathsf{r}_{b-1}^p = \mathsf{t}_{b-1}^p=\mathsf{z}^p=1, \\
& [\mathsf{r}_{j}, \, \mathsf{z}]  = [\mathsf{t}_{j}, \, \mathsf{z}]= 1, \\
& [\mathsf{r}_j, \mathsf{r}_k]= [\mathsf{t}_j, \mathsf{t}_k] = 1, \\
& [\mathsf{r}_{j}, \,\mathsf{t}_{k}] =\mathsf{z}^{- \delta_{jk}} \, \rangle.
\end{split}
\end{equation}
If $p$ is odd, this group has exponent $p^2$. 
 \end{itemize}
\end{proposition}  

\begin{remark} \label{rmk:further-commutators}
In both cases, from the relations above we deduce
\begin{equation} \label{eq:further-commutators}
[\mathsf{r}_j^{-1}, \, \mathsf{t}_k]=\mathsf{z}^{\delta_{jk}}, \quad [\mathsf{r}_j^{-1}, \, \mathsf{t}_k^{-1}]=\mathsf{z}^{-\delta_{jk}}.
\end{equation}
\end{remark}

\begin{remark} \label{rmk:center-H-G}
For both groups $\mathsf{H}_{2b+1}(\mathbb{Z}_p)$ and $\mathsf{G}_{2b+1}(\mathbb{Z}_p)$, the center is $\langle  \z  \rangle \simeq \mathbb{Z}_p$.
\end{remark}

\begin{remark}
If $p=2$, we can distinguish the two groups $\mathsf{H}_{2b+1}(\mathbb{Z}_p)$ and $\mathsf{G}_{2b+1}(\mathbb{Z}_p)$ by counting the number of elements of order $4$.
\end{remark}

\begin{remark}
The group $\mathsf{H}_{2b+1}(\mathbb{Z}_p)$ is isomorphic to the  \emph{matrix Heisenberg group} of order $p^{2b+1}$, that is, the subgroup of $\mathrm{GL}_{b+2}(\mathbb{Z}_p)$ consisting of matrices with $1$ along the diagonal and $0$ elsewhere, except for the top row and rightmost column, namely
\begin{equation} \label{eq:heis-matrices}
\mathsf{H}_{2b+1}(\mathbb{Z}_p) = \left\{
\begin{pmatrix} 1 & \mathbf{x} & z\\
{}^{\mathrm{t}} \mathbf{0} &  I_b & {}^{\mathrm{t}}\mathbf{y}  \\
0 & \mathbf{0} & 1 \\
\end{pmatrix} \; \; \bigg|  \; \; \mathbf{x}, \, \mathbf{y} \in (\mathbb{Z}_p)^b, \, z \in \mathbb{Z}_p
  \right\}.
\end{equation}
With this identification, calling $\{\mathbf{e}_1, \ldots, \mathbf{e}_b\}$ the standard basis of $(\mathbb{Z}_p)^b$, we have that:  
\begin{itemize}
\item[-] $\mathsf{r}_j$ corresponds to the matrix with $\mathbf{x}=\mathbf{0}$, $\mathbf{y}=\mathbf{e}_j$, $z=0$;
\item[-] $\mathsf{t}_j$ corresponds to the matrix with $\mathbf{x}=\mathbf{e}_j$, $\mathbf{y}=\mathbf{0}$, $z=0$; 
\item[-] $\mathsf{z}$ corresponds to the matrix with $\mathbf{x}= \mathbf{0}$, $\mathbf{y} = \textbf{0}$, $z=1$.
\end{itemize}
\end{remark}
Here is our first main result, cf. \cite[Section 3]{CaPol19}.

\begin{theorem} \label{thm:extra-special-sdks}
Let $b \geq 2$ be a positive integer and let $p$ be a prime number. If $p$ divides $b+1$, then every extra-special $p$-group $G$ of order $p^{2b+1}$ admits a diagonal double Kodaira structure of strong type $(b, \, p)$.
\end{theorem}
\begin{proof}
In both cases $G=\mathsf{H}_{2b+1}(\mathbb{Z}_p)$ and $G=\mathsf{G}_{2b+1}(\mathbb{Z}_p)$, set
\begin{equation*}
\r_{1j}=\r_{2j}:=\r_j, \quad \t_{1j}=\t_{2j}:=\t_j
\end{equation*}
and define 
\begin{equation*}
\S = \{\r_{11}, \, \t_{11}, \ldots, \r_{1b}, \, \t_{1b}, \; \r_{21}, \, \t_{21}, \ldots, \r_{2b}, \, \t_{2b}, \; \z \}. 
\end{equation*}
Since every extra-special $p$-group $G$ satisfies $[G, \, G]= Z(G)$, it suffices to check the simplified set of relations given in Remark \ref{rmk:commutators-in-the-center}. Verifying  \eqref{eq:presentation-01-simple},  \eqref{eq:presentation-1-simple} and \eqref{eq:presentation-3-simple} is immediate from the presentation of $G$ (see Proposition \ref{prop:extra-special-groups}), whereas the surface relations \eqref{eq:presentation-0-simple} follow from \eqref{eq:further-commutators} because, by assumption, we have $b=-1$ in $\mathbb{Z}_p$. Thus $\S$ provides a diagonal double Kodaira structure on $G$, that is of strong type by construction.
\end{proof}

Our next goal is to show that, if in Theorem \eqref{thm:extra-special-sdks} we drop the condition that $p$ divides $b+1$, we can still obtain some  diagonal double Kodaira structures \emph{of non-strong type} on extra-special $p$-groups of (bigger) order $p^{4b+1}$. Let us first show a couple of technical lemmas. 

\begin{lemma} \label{lem:technical}
If $b \geq 2$ is an integer and $p \geq 5$ is a prime number, we can find non-zero elements
\begin{equation}
\lambda_1, \ldots, \lambda_b, \, \mu_1, \ldots, \mu_b \in \mathbb{Z}_p
\end{equation}
such that 
\begin{equation} \label{eq:technical}
\sum_{j=1}^b \lambda_j=\sum_{j=1}^b \mu_j=1
\end{equation}
and $\lambda_j \mu_j \neq 1$ for all $j \in \{1, \ldots, b\}$.
\end{lemma}
\begin{proof}
The following simple argument is borrowed from \cite[proof of Proposition 2.16]{CaPol19}. Choose arbitrarily $\lambda_j$, with $j \in\{1,\ldots,b-1\}$, and $\mu_j$,  with $j \in\{1,\ldots,b-2\}$, such that $\lambda_j\mu_j\ne 1$ for all $j \in\{1,\ldots,b-2\}$. Then $\lambda_b$ is uniquely determined by $\lambda_b=1-\sum_{j=1}^{b-1}\lambda_j$, whereas  $\mu_{b-1}$ and $\mu_b$ are subject to the following conditions:
\begin{itemize}
\item $\mu_{b-1}+\mu_b$ is equal to a constant $c=1-\sum_{j=1}^{b-2}\mu_j$ 
\item $\mu_{b-1}\ne\lambda_{b-1}^{-1}$, $\mu_{b}\ne\lambda_{b}^{-1}$.
\end{itemize}
These requirements are in turn equivalent to $\mu_{b-1}\notin \{\lambda_{b-1}^{-1},\,c-\lambda_b^{-1}\}$. If $p \geq 5$ this can be clearly satisfied, because there are more than two non-zero elements in $\mathbb Z_p$.
\end{proof}
Now, take any anti-symmetrix matrix $A=(a_{jk})$ of order $2n$ over $\mathbb{Z}_p$, and consider the finitely presented groups
\begin{equation} \label{eq:H(A)}
\begin{split}
\mathsf{H}(A) = \langle \,  \mathsf{x}_1, \ldots, \mathsf{x}_{2n}, \, \mathsf{z} \; | \; & \mathsf{x}_{1}^p = \ldots =\mathsf{x}_{2n}^p=\mathsf{z}^p=1, \\ 
& [\x_1, \, \z] = \ldots = [\x_{2n}, \, \z]=1,\\
& [\x_j, \, \x_k] =\z^{a_{jk}} \, \rangle,
\end{split}
\end{equation}

\begin{equation} \label{eq:G(A)}
\begin{split}
\mathsf{G}(A) = \langle \,  \mathsf{x}_1, \ldots, \mathsf{x}_{2n}, \, \mathsf{z} \; | \; & \mathsf{x}_{1}^p  = \ldots =\mathsf{x}_{2n-2}^p=\mathsf{z}^p=1, \\ 
& \x_{2n-1}^p  =\x_{2n}^p= \z, \\
 & [\x_1, \, \z]  = \ldots = [\x_{2n}, \, \z]=1,\\
& [\x_j, \, \x_k]  =\z^{a_{jk}} \, \rangle,
\end{split}
\end{equation}
where the exponent in $\z^{a_{jk}}$ stands for any representative in $\mathbb{Z}$ of $a_{jk} \in \mathbb{Z}_p$.
\medskip

Recall that, given three elements $a, \, b, \, c$ in a group $G$, we have the commutator relation $[a, \, bc]=[a, \, b] b [a, \, c] b^{-1}$. Since  all commutators in $\mathsf{H}(A)$ are central, we get 
\begin{equation} \label{eq:central-in-H(A)}
[a, \, bc]=[a, \,b]  [a, \, c] \quad \textrm{for all} \; a, \, b, \, c \in \mathsf{H}(A),
\end{equation}
and similarly for $\mathsf{G}(A)$.
\begin{lemma} \label{lem:identify-heisenberg}
If $\det A \neq 0$, then the following holds$:$
\begin{itemize}
\item $\mathsf{H}(A)$ is an extra-special $p$-group of order $p^{2n+1}$ and exponent $p$. In particular, it is isomorphic to $\mathsf{H}_{2n+1}(\mathbb{Z}_p)$;
\item $\mathsf{G}(A)$ is an extra-special $p$-group of order $p^{2n+1}$ and exponent $p^2$. In particular, it is isomorphic to $\mathsf{G}_{2n+1}(\mathbb{Z}_p)$.
\end{itemize}
\end{lemma}
\begin{proof}
We prove only the first point, the second being similar. The commutator relations in \eqref{eq:H(A)} show that every element of $\mathsf{H}(A)$ can be written in the form $\mathsf{x}_1^{t_1} \ldots \mathsf{x}_{2n}^{t_{2n}}$, with $t_1,\ldots, t_{2n} \in \mathbb{Z}$. Since $\mathsf{x}_j$ has order $p$ and $[\x_j, \, \x_k]$ is central, it follows that $\mathsf{H}(A)$ has exponent $p$. 

The quotient of $\mathsf{H}(A)$ by the central subgroup $\langle \z \rangle$ is an elementary abelian $p$-group of order $p^{2n}$. Therefore the only remaining issue is check that the center of $\mathsf{H}(A)$ is precisely $\langle \z \rangle$, and no larger.  To this purpose, it suffices to check that an element of the form $\mathsf{x}_1^{t_1} \ldots \mathsf{x}_{2n}^{t_{2n}}$  is central if and only if all the $t_j$ are zero. 
By using \eqref{eq:H(A)} and \eqref{eq:central-in-H(A)}, we get 
\begin{equation}
\begin{split}
[\mathsf{x}_k, \, \mathsf{x}_1^{t_1} \ldots 
\mathsf{x}_{2n}^{t_{2n}}]& =[\x_k, \, \x_1]^{t_1} \ldots 
[\x_k, \, \x_{2n}]^{t_{2n}} \\
& = \z^{a_{k1}t_1 + \ldots + a_{k2n} t_{2n} }.
\end{split}
\end{equation}
It follows that  $\mathsf{x}_1^{t_1} \ldots \mathsf{x}_{2n}^{t_{2n}}$ is central if and only if we have 
\begin{equation}
a_{k1}t_1 + \ldots + a_{k2n} t_{2n} =0, \quad k=1,\ldots, 2n.
\end{equation}
This is a homogeneous system of linear equations in the variables $t_1, \ldots, t_{2n}$ and whose coefficient matrix is $A$. Being $A$ non-singular by assumption, there is only the trivial solution $t_1=\ldots = t_{2n}=0$.  
\end{proof}

We are now in a position to prove our second main result, cf. \cite[Section 2]{CaPol19}.

\begin{theorem} \label{thm:non-strong}
If $b \geq 2$ is a positive integer and $p \geq 5$ is a prime number, then every extra-special $p$-group $G$ of order $p^{4b+1}$ admits a diagonal double Kodaira structure of non-strong type $(b, \, p)$.
\end{theorem}
\begin{proof}
Again, we treat in detail the case $G=\mathsf{H}_{4b+1}(\mathbb{Z}_p)$; the proof for $G=\mathsf{G}_{4b+1}(\mathbb{Z}_p)$ is similar. Let us consider the anti-symmetric matrix 
\begin{equation} \label{eq:omega}
\Omega_b=
\begin{pmatrix}
L_b & J_b \\
J_b & M_b
\end{pmatrix} \in \mathrm{Mat}_{4b}(\mathbb{Z}_p),
\end{equation}
where the blocks are the elements of $\mathrm{Mat}_{2b}(\mathbb{Z}_p)$ given by
\begin{equation} \label{eq:L}
L_b=\begin{pmatrix}
\begin{matrix}0 & \lambda_1 \\ - \lambda_1 & 0\end{matrix} & & 0 \\
 & \ddots & \\
0 & & \begin{matrix}0 & \lambda_b \\ - \lambda_b & 0
\end{matrix}
\end{pmatrix} \quad
M_b=\begin{pmatrix}
\begin{matrix}0 & \mu_1 \\ - \mu_1 & 0\end{matrix} & & 0 \\
 & \ddots & \\
0 & & \begin{matrix}0 & \mu_b \\ - \mu_b & 0
\end{matrix}
\end{pmatrix}
\end{equation}
\begin{equation} \label{eq:J}
J_b=\begin{pmatrix}
\begin{matrix}
0 & -1 \\ 1 & \; \; 0\end{matrix} & & 0 \\
 & \ddots & \\
0 & &
\begin{matrix}0 & -1 \\ 1 & \; \; 0
\end{matrix}
\end{pmatrix}
\end{equation}
and $\lambda_1, \ldots, \lambda_b$, $\mu_1, \ldots, \mu_b$ are as in Lemma \ref{lem:technical}. We have 
\begin{equation} \label{eq:determinant-Omega}
\det \Omega_b = (1-\lambda_1 \mu_1)^2 (1- \lambda_2 \mu_2)^2 \cdots (1-\lambda_b \mu_b)^2 \neq 0
\end{equation}
and so, by Lemma \ref{lem:identify-heisenberg}, we infer that $\mathsf{H}(\Omega_b)$ is isomorphic to $\mathsf{H}_{4b+1}(\mathbb{Z}_p)$. By definition, the group $\mathsf{H}(\Omega_b)$ is generated by a set of  $4b+1$ elements 
\begin{equation}
\S = \{\r_{11}, \, \t_{11}, \ldots, \r_{1b}, \, \t_{1b}, \; \r_{21}, \, \t_{21}, \ldots, \r_{2b}, \, \t_{2b}, \; \z \}
\end{equation}
subject to the relations
\begin{equation} \label{eq:rel-heis-g}
\begin{split}
\mathsf{r}_{1j}^p & = \mathsf{t}_{1j}^p= \mathsf{r}_{2j}^p = \mathsf{t}_{2j}^p= \mathsf{z}^p=1, \\
[\mathsf{r}_{1j}, \, \z] & = [\mathsf{t}_{1j}, \, \z]= [\mathsf{r}_{2j}, \, \z] = [\mathsf{t}_{2j}, \, \z]=1, \\
[\mathsf{r}_{1j}, \, \mathsf{r}_{1k}]& =[\mathsf{t}_{1j}, \, \mathsf{t}_{1k}]=1, \\
[\mathsf{r}_{1j}, \, \mathsf{r}_{2k}]& =[\mathsf{t}_{1j}, \, \mathsf{t}_{2k}]=1, \\
[\mathsf{r}_{2j}, \, \mathsf{r}_{2k}]& =[\mathsf{t}_{2j}, \, \mathsf{t}_{2k}]=1, \\
[\mathsf{r}_{1j}, \,\mathsf{t}_{1 k}]& =\mathsf{z}^{\delta_{jk} \, \lambda_j}, \\
[\mathsf{r}_{2j}, \,\mathsf{t}_{2 k}]& =\mathsf{z}^{\delta_{jk} \, \mu_j}, \\
[\mathsf{r}_{1j}, \,\mathsf{t}_{2 k}]& = [\mathsf{r}_{2j}, \,\mathsf{t}_{1 k}]=\mathsf{z}^{-\delta_{jk}}.
\end{split}
\end{equation}
Using \eqref{eq:technical}, we can check that the two surface relations \eqref{eq:presentation-0-simple} are satisfied. Since the remaining relations \eqref{eq:presentation-01-simple}, \eqref{eq:presentation-1-simple} and \eqref{eq:presentation-3-simple} clearly hold, it follows that $\S$ provides a diagonal double Kodaira structure of type $(b, \, p)$ on $\mathsf{H}(\Omega_b)$, and so a diagonal double Kodaira structure of the same type on the isomorphic group $\mathsf{H}_{4b+1}(\mathbb{Z}_p)$. Such a structure is not strong, because the two subgroups $K_1=\langle  \r_{11}, \, \t_{11}, \ldots, \r_{1b}, \, \t_{1b}, \; \z  \rangle$ and $K_2=\langle  \r_{21}, \, \t_{21}, \ldots, \r_{2b}, \, \t_{2b}, \; \z  \rangle$ are isomorphic to $\mathsf{H}_{2b+1}(\mathbb{Z}_p)$, hence they both have index $p^{2b}$ in $G$.
\end{proof}

\begin{remark} \label{rmk:no-lemma-if-p=3}
The conclusion of Lemma \ref{lem:technical} is false when $p\leq 3$. If $p=2$, this follows immediately from the fact that there exists a unique non-zero element in $\mathbb{Z}_2$. If $p=3$, two non-zero elements $\lambda_i$, $\mu_i \in \mathbb{Z}_3$ satisfy $\lambda_i \mu_i \neq 1$ if and only if  $\lambda_i = -\mu_i$, so \eqref{eq:technical} cannot hold. This shows that, if Theorem \ref{thm:non-strong} is also true for $p \leq 3$, then it must be proved in a different way.
\end{remark}

\begin{remark}
The existence of a diagonal double Kodaira structure of non-strong type on $\mathsf{H}_{4b+1}(\mathbb{Z}_p)$ was first showed in \cite[Section 2]{CaPol19}, although we did not use this terminology; the original proof relies on some group-cohomological results related to the structure the cohomology algebra $H^*(\Sigma_b \times \Sigma_b - \Delta, \, \mathbb{Z}_p)$, where $\Sigma_b$ is a real surface of genus $b$ and $\Delta \subset \Sigma_b \times \Sigma_b$ is the diagonal. Besides, such a proof does not use Lemma \ref{lem:identify-heisenberg}, but an equivalent statement coming from the identification of $\mathsf{H}_{4b+1}(\mathbb{Z}_p)$ with the so-called \emph{symplectic Heisenberg group} $\mathsf{Heis}(V, \, \omega)$, where $V=H_1(\Sigma_b \times \Sigma_b - \Delta, \, \mathbb{Z}_p) \simeq (\mathbb{Z}_p)^{4b}$ and $\omega$ is any symplectic form on $V$.

Then, assuming that $p$ divides $b+1$, in \cite[Section 3]{CaPol19} we deduced the existence of a diagonal double Kodaira structure of strong type on $\mathsf{H}_{2b+1}(\mathbb{Z}_p)$ by setting
\begin{equation} 
\lambda_1=\ldots = \lambda_b=\mu_1=\ldots =  \mu_b = -1.
\end{equation}
Indeed, this yields a diagonal double Kodaira structure on a ``degenerate" Heisenberg group of order $p^{4b+1}$ (in this case $\det \Omega_b = 0$), admitting the group $\mathsf{H}_{2b+1}(\mathbb{Z}_p)$ as a quotient.

In this note we adopted, instead, a purely group-theoretical approach; it is less geometric but shorter than the original one and it naturally yields new results, namely the existence of  diagonal double Kodaira structures on the extra-special $p$-groups of exponent $p^2$.

\end{remark}

\section{Geometric interpretation: from diagonal double Kodaira structures to diagonal double Kodaira fibrations}

For more details on the basic definitions and results of this section, we refer the reader to the Introduction and to  \cite{CaPol19}, especially Sections 1 and 3. Recall that a \emph{Kodaira fibration} is a smooth, connected holomorphic fibration $f_1 \colon S \to B_1$, where $S$ is a compact complex surface and $B_1$ is a compact complex curve, which is not isotrivial. The genus $b_1:=g(B_1)$ is called the \emph{base genus} of the fibration, whereas the genus $g:=g(F)$, where $F$ is any fibre, is called the \emph{fibre genus}.
\begin{definition} \label{def:double-kodaira}
A \emph{double Kodaira surface} is a compact complex surface $S$, endowed with a \emph{double Kodaira fibration}, namely a surjective, holomorphic map $f \colon S \to B_1 \times B_2$ yielding, by composition with the natural projections, two Kodaira fibrations $f_i \colon S \to B_i$, $i=1, \,2$.
\end{definition}
The aim of this section is to show how the existence of  diagonal double Kodaira structures is equivalent to the existence of some special double Kodaira fibrations, that we call \emph{diagonal double Kodaira fibrations}. Looking at 
Gon\c{c}alves-Guaschi's presentation of surface pure braid groups, see \cite[Theorem 7]{GG04}, \cite[Theorem 1.7]{CaPol19}, we see that a finite group $G$ admits a diagonal double Kodaira structure $\S$ of type $(b, \, n)$ if and only if there is a surjective group homomorphism
\begin{equation} \label{eq:varphi}
\varphi \colon \mathsf{P}_2(\Sigma_b) \to G,
\end{equation} 
such that $\z:=\varphi(A_{12})$ has order $n$. 
Here $\mathsf{P}_2(\Sigma_b)$ is the pure braid group of genus $b$ on two strands, which is isomorphic to the fundamental group $\pi_1(\Sigma_b \times \Sigma_b - \Delta, \, (p_1, \, p_2))$  of the configuration space  of two ordered points on a real surface of genus $b$, and the generator $A_{12}$ is the homotopy class in $\Sigma_b \times \Sigma_b - \Delta$ of a loop in $\Sigma_b \times \Sigma_b$ that ``winds once" around the diagonal $\Delta$. 

With a slight abuse of notation, in the sequel we will use the symbol $\Sigma_b$ to indicate both a smooth complex curve of genus $b$ and its underlying real surface. By using Grauert-Remmert's extension theorem together with Serre's GAGA, the group epimorphism $\varphi$ gives the existence of a smooth, complex, projective surface $S$ endowed with a Galois cover 
\begin{equation}
\mathbf{f} \colon S \to \Sigma_b \times \Sigma_b,
\end{equation}
with Galois group $G$ and branched precisely over $\Delta$ with branching order $n$, see \cite[Proposition 3.4]{CaPol19}. 

The braid group $\mathsf{P}_2(\Sigma_b)$ is the middle term of two short exact sequences 
\begin{equation} \label{eq:oggi}
1\to \pi_1(\Sigma_b -\{p_i\}, \, p_j) \to \mathsf{P}_2(\Sigma_b) \to \pi_1(\Sigma_b, \, p_i) \to 1, 
\end{equation}
where $\{i, \, j\} = \{1, \, 2\}$, induced by the two natural projections of pointed topological spaces $(\Sigma_b \times \Sigma_b - \Delta, \, (p_1, \, p_2)) \to (\Sigma_b, \, p_i)$. Composing the left homomorphism in \eqref{eq:oggi} with $\varphi \colon \mathsf{P}_2(\Sigma_b) \to G$, we get two homomorphisms
\begin{equation} \label{eq:varphi-i}
\varphi_1 \colon \pi_1(\Sigma_b -\{p_2\}, \, p_1) \to G, \quad \varphi_2 \colon \pi_1(\Sigma_b -\{p_1\}, \, p_2) \to G,
\end{equation}
whose image equals $K_1$ and $K_2$, respectively. By construction, these are the homomorphisms induced by the restrictions $\mathbf{f}_i \colon \Gamma_i \to \Sigma_b$  of the Galois cover $\mathbf{f} \colon S \to \Sigma_b \times \Sigma_b$ to the fibres of the two natural projections $\pi_i \colon \Sigma_b \times \Sigma_b \to \Sigma_b$. Since $\Delta$ intersects transversally at a single point all the fibres of the natural projections, it follows that both such restrictions are branched at precisely one point, and the number of connected components of the smooth curve $\Gamma_i \subset S$ equals the index $m_i:=[G : K_i]$ of $K_i$ in $G$.

So, taking the Stein factorizations of the compositions $\pi_i \circ \mathbf{f} \colon S \to \Sigma_b$ as in the diagram below
\begin{equation} \label{dia:Stein-Kodaira-gi}
\begin{tikzcd}
  S \ar{rr}{\pi_i \circ \mathbf{f}}  \ar{dr}{f_i} & & \Sigma_b  \\
   & \Sigma_{b_i} \ar{ur}{\theta_i} &
  \end{tikzcd}
\end{equation}
we obtain two distinct Kodaira fibrations $f_i \colon S \to \Sigma_{b_i}$, hence a double Kodaira fibration by considering the product morphism  \begin{equation}
f=f_1 \times f_2 \colon S \to \Sigma_{b_1} \times \Sigma_{b_2}. 
\end{equation}
\begin{definition} \label{def:diagonal-double-kodaira}
We call $f \colon S \to \Sigma_{b_1} \times \Sigma_{b_2}$ the \emph{diagonal double Kodaira fibration} associated with the diagonal double Kodaira structure $\S$ on the finite group $G$. Conversely, we will say that a double Kodaira fibration $f \colon S \to \Sigma_{b_1} \times \Sigma_{b_2}$ is \emph{of diagonal type} $(b, \, n)$ if there exists a finite group $G$ and a diagonal double Kodaira structure $\S$ of type $(b, \, n)$ on it such that $f$ is associated with $\S$.  
\end{definition}
One can wonder whether all double Kodaira fibrations are of diagonal type; the answer is negative, as we will show in Example \ref{ex:non-diagonal}, see also Proposition \ref{prop:slope-diagonal} and Remark \ref{rmk:6-4sqrt2}.

\medskip

Since the morphism $\theta_i \colon \Sigma_{b_i} \to \Sigma_b$ is \'{e}tale of degree $m_i$, by using the Hurwitz formula we obtain
\begin{equation} \label{eq:expression-gi}
b_1 -1 =m_1(b-1), \quad  b_2 -1 =m_2(b-1).
\end{equation}
Moreover, the fibre genera $g_1$, $g_2$ of the Kodaira fibrations $f_1 \colon S \to \Sigma_{b_1}$, $f_2 \colon S \to \Sigma_{b_2}$ are computed by the formulae
\begin{equation} \label{eq:expression-gFi}
2g_1-2 = \frac{|G|}{m_1} (2b-2 + \mathfrak{n} ), \quad 2g_2-2 = \frac{|G|}{m_2} \left( 2b-2 + \mathfrak{n} \right),
\end{equation}
where $\mathfrak{n}:= 1 - 1/n$. Finally, the surface $S$ fits into a diagram
\begin{equation} \label{dia:degree-f-general}
\begin{tikzcd}
  S \ar{rr}{\mathbf{f}}  \ar{dr}{f} & & \Sigma_b \times \Sigma_b  \\
   & \Sigma_{b_1} \times \Sigma_{b_2} \ar[ur, "\theta_1 \times \theta_2"{sloped, anchor=south}] &
  \end{tikzcd}
\end{equation}
so that the diagonal double Kodaira fibration $f \colon S \to  \Sigma_{b_1} \times \Sigma_{b_2}$ is a finite cover of degree $\frac{|G|}{m_1m_2}$, branched precisely over the curve
\begin{equation} \label{eq:branching-f}
(\theta_1 \times \theta_2)^{-1}(\Delta)=\Sigma_{b_1} \times_{\Sigma_b} \Sigma_{b_2}.
\end{equation}
Such a curve is always smooth, being the preimage of a smooth divisor via an \'{e}tale morphism. However, it is reducible in general, see \cite[Proposition 3.11]{CaPol19}. The invariants of $S$ can be now computed as follows, see \cite[Proposition 3.8]{CaPol19}.

\begin{proposition} \label{prop:invariant-S-G}
Let $f \colon S \to \Sigma_{b_1} \times \Sigma_{b_2}$ be the diagonal double Kodaira fibration associated with a diagonal double Kodaira structure $\S$ of type $(b, \, n)$ on a finite group $G$. Then we have
\begin{equation} \label{eq:invariants-S-G}
\begin{split}
c_1^2(S) & = |G|\,(2b-2) ( 4b-4 + 4 \mathfrak{n} - \mathfrak{n}^2 ) \\
c_2(S) & =   |G|\,(2b-2) (2b-2 + \mathfrak{n}).
\end{split}
\end{equation}
As a consequence, the slope and the signature of $S$ can be expressed as
\begin{equation} \label{eq:slope-signature-S-G}
\begin{split}
\nu(S) & = \frac{c_1^2(S)}{c_2(S)} = 2+ \frac{2 \mathfrak{n} - \mathfrak{n}^2}{2b-2 + \mathfrak{n} } \\
\sigma(S) & = \frac{1}{3}\left(c_1^2(S) - 2 c_2(S) \right) = \frac{1}{3} \, |G| \, (2b-2)(2 \mathfrak{n} - \mathfrak{n}^2),
\end{split}
\end{equation}
where $\mathfrak{n}=1-1/n$.
\end{proposition}

\begin{remark} \label{rmk:sdks-characterization}
By definition, $\S$ is a diagonal double Kodaira structure of strong type if and only if $m_1=m_2=1$, that in turn implies $b_1=b_2=b$, i.e., $f=\mathbf{f}$. In other words, $\S$ is of strong type if and only if no Stein factorization as in \eqref{dia:Stein-Kodaira-gi} is needed or, equivalently, if and only if the Galois cover $\mathbf{f}\colon S \to \Sigma_b \times \Sigma_b$ induced by \eqref{eq:varphi} is already a double Kodaira fibration, branched on the diagonal $\Delta \subset \Sigma_b \times \Sigma_b$. 
\end{remark}

We can now specialize the previous results, by taking as $G$ an extra-special $p$-group and using what we have proved in Section \ref{sec:extra-special}. Let $\boldsymbol{\upomega} \colon \mathbb{N} \to \mathbb{N}$ be the arithmetic function counting the number of distinct prime factors of a positive integer, see \cite[p.335]{HarWr08}. The following is  \cite[Corollary 3.18]{CaPol19}.

\begin{theorem} \label{thm:lim-sup}
Let $\Sigma_b$ be any smooth curve of genus $b$. Then there exists a double Kodaira fibration $f \colon S \to \Sigma_b \times \Sigma_b$. Moreover, denoting by $\kappa(b)$ the number of such fibrations, we have
\begin{equation*}
\kappa(b) \geq \boldsymbol{\upomega}(b+1).
\end{equation*}
In particular,
\begin{equation*}
\limsup_{b \rightarrow + \infty} \kappa(b) = + \infty.
\end{equation*}
\end{theorem}
\begin{proof}
Given a prime number $p$ dividing $b+1$, every extra-special $p$-group $G$ of order $p^{2b+1}$ admits a diagonal double Kodaira structure of strong type $(b, \, p)$, see Theorem \ref{thm:extra-special-sdks}, and this gives in turn a diagonal double Kodaira fibration $f \colon S \to \Sigma_b \times \Sigma_b$, see Remark \ref{rmk:sdks-characterization}. Two different prime divisors of $b+1$ give rise to two non-homeomorphic double Kodaira surfaces, because the corresponding signatures are different (use \eqref{eq:slope-signature-S-G} with $n=p$ and note that, for fixed $b$, the function expressing $\sigma(S)$ is strictly increasing in $p$). Since the number of distinct prime factors of $b+1$ can be arbitrarily large when $b$ goes to infinity, the last statement follows.
\end{proof}

The case $b=2, \, p=3$ is particularly interesting. In fact, it provides (to our knowledge) the first 
``double  solution" to a problem (posed by Geoff Mess) from Kirby's problem list in low-dimensional topology (\cite[Problem 2.18A]{Kir97}), asking what is the smallest number $b$ for which there exists a real surface bundle over a surface with base genus $b$ and non-zero signature, see \cite[Proposition 3.19]{CaPol19}.

\begin{theorem} \label{thm:Kirby}
Let $S$ be the diagonal double Kodaira surface associated with a diagonal double Kodaira structure of strong type $(2, \, 3)$ on an extra-special $3$-group $G$ of order $3^5$. Then the real manifold $X$ underlying $S$ is a closed, orientable $4$-manifold of signature $144$ that can be realized as a real surface bundle over a surface of genus $2$, with fibre  genus $325$, in two different ways.
\end{theorem}
This naturally leads to the following interesting problem, see \cite[Question 3.20]{CaPol19}.
\begin{question} \label{q:minimal-values}
What are the minimal possible fibre genus $f_{\mathrm{min}}$ and the minimum possible signature $\sigma_{\mathrm{min}}$ for a double Kodaira fibration $S \to \Sigma_2 \times \Sigma_2$? 
\end{question}
Note that Theorem \ref{thm:Kirby} implies $f_{\mathrm{min}} \leq 325$ and $\sigma_{\mathrm{min}} \leq 144$.
\medskip

Let us show now how to use our methods in order to obtain double Kodaira fibrations with slope strictly higher than $2+1/3$. Fix $b=2$ and let $p \geq 5$ be a prime number. Then every
extra-special $p$-group $G$ of order $p^{4b+1}=p^9$ admits a diagonal double Kodaira structure $\S$ of non-strong type $(2, \, p)$ and such that $m_1=m_2=p^{2b}$, see Theorem \ref{thm:non-strong}. Setting $b':=p^{4}+1$, cf. equations \eqref{eq:expression-gi}, we obtain the following particular case of \cite[Proposition 3.12]{CaPol19}.

\begin{theorem} \label{thm:slope}
Let $f \colon S_{2, \, p} \to \Sigma_{b'} \times \Sigma_{b'}$ be the diagonal double Kodaira fibration associated with a diagonal double Kodaira structure of non-strong type $(2, \, p)$ on an extra-special $p$-group $G$ of order $p^9$. Then the maximum slope $\nu(S_{2, \, p})$ is attained for precisely two values of $p$, namely
\begin{equation}
\nu(S_{2, \, 5}) = \nu(S_{2, \, 7})= 2 + \frac{12}{35}.
\end{equation}
Furthermore, $\nu(S_{2, \, p}) > 2 + 1/3$ for all $p \geq 5$. More precisely, if $p \geq 7$ the function $\nu(S_{2, \, p})$ is strictly decreasing and
\begin{equation*}
\lim_{p \rightarrow +\infty} \nu(S_{2, \, p}) = 2 + \frac{1}{3}.
\end{equation*}
\end{theorem}   

\begin{remark} \label{rmk:record-slope}
The original examples by Atiyah, Hirzebruch and Kodaira have slope lying in the interval $(2, 2 + 1/3]$, see \cite[p. 221]{BHPV03}. Our construction provides an infinite family of Kodaira fibred surfaces such that $2+1/3 < \nu(S) \leq 2 + 12/35$, maintaining at the same time a complete control on both the base genus and the signature.
By contrast, the ``tautological construction" used in \cite{CatRol09} yields a higher slope than ours, namely $2+ 2/3$, but it involves an {\'e}tale pullback ``of sufficiently large degree'', that completely loses control on the other quantities. 
\end{remark}
\begin{remark} \label{rmk:record-slope-bis}
By Liu's inequality (see \cite{Liu96}), every Kodaira fibred surface $S$ satisfies $\nu(S) <3$. The value $\nu=2 + 2/3$ is the current record for the slope, in particular it is  unknown whether the slope of a Kodaira fibred surface can be arbitrarily close to $3$. 
\end{remark}

Finally, let us show that there exist double Kodaira fibrations that are not of diagonal type.

\begin{example} \label{ex:non-diagonal}
Take any double Kodaira fibration $f \colon S \to \Sigma_{b_1} \times \Sigma_{b_2}$ with $b_2=2$ and $\nu(S)=2+1/2$, see for instance \cite[Examples 6.3 and 6.6 of Table 3]{LLR17}. We claim that such a $f$ cannot be of diagonal type. In fact, assume by contradiction that $f$ is associated with a diagonal double Kodaira structure of type $(b, \, n)$ on a finite group $G$. Then, by using the second equation in \eqref{eq:expression-gi}, we obtain $2-1=m_2(b-1)$, hence $b=2$. Substituting in the slope expression in \eqref{eq:slope-signature-S-G}, we get
\begin{equation}
\frac{1}{2}=\frac{2 \mathfrak{n} - \mathfrak{n}^2}{2 + \mathfrak{n} },
\end{equation}
or, equivalently, $n^2-n+2=0$, that has no integer solutions.
  \end{example}

In fact, Example \ref{ex:non-diagonal} is an instance of the following, more general result.

\begin{proposition} \label{prop:slope-diagonal}
Let $f \colon S \to \Sigma_{b_1} \times \Sigma_{b_2}$ be a double Kodaira fibration of diagonal type $(b, \, n)$. Then we have $\nu(S)=2+s$, where $s$ is a strictly positive rational number such that $(s+2)^2-8bs$ is a perfect square in $\mathbb{Q}$. As a consequence, we obtain $s < 6-4\sqrt{2}$.
\end{proposition} 
\begin{proof}
By definition $\nu(S)$ is a rational number, and moreover $\nu(S)>2$ because of Arakelov inequality, 
see \cite{Be82} . So we can write $\nu(S)=2+s$, with $s>0$. Since we are assuming that $S$ is associated with a diagonal double Kodaira structure of type $(b, \, n)$, the slope identity in \eqref{eq:slope-signature-S-G} yields
\begin{equation}
s=\frac{2 \mathfrak{n} - \mathfrak{n}^2}{2b-2 + \mathfrak{n} },
\end{equation}
or, equivalently, 
\begin{equation}
(2bs-s-1)n^2-sn+1=0. 
\end{equation}
The discriminant of this quadratic equation is $(s+2)^2-8bs$, and this quantity must be a perfect square in $\mathbb{Q}$ because $n$ is an integer number. In particular, we have $(s+2)^2 \geq 8bs$, that is, 
\begin{equation} \label{eq:b-s}
2 \leq b \leq \frac{(s+2)^2}{8s}.
\end{equation}
From \eqref{eq:b-s} we deduce the inequality $(s+2)^2-16s \geq 0$; since Remark \ref{rmk:record-slope-bis} gives $s <1$, we infer $s < 6-4\sqrt{2}$.
\end{proof}

\begin{remark} \label{rmk:6-4sqrt2}
Since $6-4\sqrt{2} = 0. 3431...$ and $12/35 = 0. 3428...$, we see that the surfaces $S_{2, \, 5}$ and $S_{2, \, 7}$, described in Theorem \textrm{\ref{thm:slope}}, ``almost maximize" the slope of a double Kodaira fibration of diagonal type. In fact, the upper bound $s < 6-4\sqrt{2}$ shows that high slope examples, like Catanese-Rollenske's one for which $s=2/3$, are out of reach of the methods of this paper. 
\end{remark}

\end{document}